\documentclass[preprint]{elsarticle}

\usepackage{amssymb}
 \usepackage{amsthm}
\usepackage{amsmath}
\usepackage{tikz}

\usetikzlibrary{decorations.pathreplacing}
\usetikzlibrary{decorations.pathmorphing}

\theoremstyle{plain}
\newtheorem{theorem}{Theorem}[section]

\newtheorem{construction}[theorem]{Construction}
\newtheorem{example}[theorem]{Example}
\newtheorem{lemma}[theorem]{Lemma}

\newtheorem{definition}[theorem]{Definition}

\newtheorem{question}[theorem]{Question}

\journal{---}

\begin{document}

\begin{frontmatter}

\title{On the intersection of three or four transversals of the back circulant latin square $B_n$}

\author{Trent Gregory Marbach\fnref{label2}}
\fntext[label2]{The research for this paper was also supported in part by the Australian Research Council (grant number DP1092868)}
\ead{trent.marbach@uqconenct.edu.au}
\address{Department of Mathematics\\
The University of Queensland\\
Brisbane 4072\\
Australia}

\begin{abstract}
A paper by Cavenagh and Wanless \cite{NumbTran} determined the possible intersection of any two transversals of the back circulant latin square $B_n$, and used the result to completely determine the spectrum for $2$-way $k$-homogeneous latin trades. 
We generalize this problem to the intersection of $\mu$ transversals of $B_n$ such that the transversals intersect stably  (that is, the intersection of any pair of transversals is independent of the choice of the pair) and show that these structures can be used to construct  $\mu$-way $k$-homogeneous circulant latin trades of odd order.
We provide a number of basic existence and non-existence results for $\mu$ transversals of $B_n$ that intersect stably, as well as the results of a computational search for small $n$. 
This is followed by the principal results of this paper; a construction that covers a large portion of the spectrum when $n$ is sufficiently large, which requires certain base designs. These base designs are provided in the cases $\mu=3,4$, which were found by a computational search. 
We use this result to find the existence of $\mu$-way $k$-homogeneous circulant latin trades of odd order, for $\mu=3,4$.
\end{abstract}


\end{frontmatter}

\section{Introduction}

A natural question in combinatorics  asks how two distinct examples of a certain combinatorial structure may intersect, which has been investigated for a large variety of different structures.
An extension of this is to consider the $\mu$-way intersections of the structures, and work has been done taking the underlying structure to be Steiner Triple Systems in  \cite{3WaySTS}, $m$-cycle systems in \cite{muWayMcycle}, and latin squares in \cite{3waySpectrum} and \cite{muWayVolumes}.

There has been an investigation into the possible intersection size of two transversals of the back circulant latin square \cite{NumbTran}, and so in a similar fashion we generalize from the intersection of two transversals to the intersection of a set of $\mu$ transversals.

The problem that this paper investigates is as follows:

\begin{question} \label{Qn:main}
For what $t$ does there exist a set of $\mu$ transversals of the back circulant latin square of order $n$, such that each pair of transversals intersect precisely in the same $t$ points?
\end{question}
 
These transversals can be used to construct $3$-way $k$-homogeneous latin trades of odd order, which will further extend our knowledge towards answering Question $1$ of \cite{3way}.

\subsection{Definitions}

A \emph{partial latin square} is an $n \times n$ array of cells, each being either empty or filled with  one of $n$ symbols such that each symbol appears at most once in each row and in each column. 
A \emph{latin square} is a partial latin square with no empty cells.   
We are able to think of a (partial) latin square as a set of triples; if a (partial) latin square, $L$, has the cell of row $r$ and column $c$ filled with symbol $e$, we will write $(r,c,e) \in L$. 
This is commonly called orthogonal array notation. 
In this paper, we write an interval of integers as $[a,b]=\{a, \ldots, b\}$. We index the rows, columns, and symbols of a latin square by $[0,n-1]$.
We will sometimes reference rows, columns, and symbols with indexes that are greater than $n-1$, by which we will always mean the representation of this index modulo $n$.

A \emph{diagonal} of a latin square $L$ is a set of $n$ cells of $L$ such that each row and each column is represented in the set of cells precisely once. 
A \emph{transversal} of a latin square is a diagonal that also has each symbol represented precisely once in the diagonal. 
See \cite{TranLSSurvey2} for a survey of transversals in latin squares.

A commonly studied latin square is the back circulant latin square, which is defined as $B_n = \{(r,c,r+c)  \mid r,c \in [0,n-1]\}$. 
The latin squares $B_n$ have a strong connection to diagonally cyclic latin squares, and are often used to prove facts about latin squares in general.  

A transversal of $B_n$ is equivalent to a diagonally cyclic latin square of order $n$, a complete mapping of the cyclic group of order $n$, an orthomorphism of the cyclic group of order $n$, a magic juggling sequence of period $n$, and a placement of $n$ non-attacking semi-queens of an $n \times n$ toroidal chessboard (see \cite{DiagCycLS}, \cite{Juggling}, \cite{Complete3Trans}).

Throughout this paper, we assume $n$ is odd, as it is well known that $B_n$ contains no transversals for any even $n$. 
The possible intersection sizes of any two transversals of $B_n$ has been determined:
\begin{theorem}\cite{NumbTran}
For each odd $n$, there exists a pair of transversals of $B_n$ that intersect in $t$ cells, when $n \neq 5$ for $t \in \{0, \ldots, n-3\}\cup \{n\}$, and when $n=5$ for $t \in \{0,1,5\}$. 
\end{theorem}

We consider a generalization of such intersections of pairs of transversals to the intersection of $\mu$ transversals.
\begin{definition} \label{maindef}
A collection of $\mu$ transversals $T_1, \ldots, T_\mu$ intersect \emph{stably} in $t$ points if there is a set $S\subseteq [0,n-1]^2$ such that  $S=\cap_{i=1}^{\mu} T_i$ and $\emptyset= (T_i \cap T_j )\setminus S$ for each $1 \leq i < j \leq \mu$. 
\end{definition}

Informally, if there is a cell $(i,j,k)\in S$, then $(i,j,k)$ appears in each transversal  $T_1, \ldots, T_\mu$. 
If there is a cell $(i',j',k') \in T_{\alpha}$ with $(i',j',k') \notin S$, then no other transversal contain $(i',j',k')$. 

Then Question \ref{Qn:main} is asking for what $t$ does there exist a set of $\mu$ transversals of $B_n$ that intersect stably in $t$ points.
The main results of this paper are the following two theorems:
\begin{theorem}\label{overall3}
For odd $n \geq 33$, let $I$, $I'$, $d$ and $d'$ be the unique integers such that $n= 18I +9 + 2d$ and $n = 22I'+11+2d'$, $I,I' \geq 1$, $0 \leq d < 9$ and $0 \leq d' < 11$. Then there exist three transversals of $B_n$ that intersect stably in $t$ points for $t\in [min(3+d',d), n] \setminus [n-5,n-1]$ except, perhaps, when:
\begin{itemize}
\item $n=51$ and $t=29$,
\item $n=53$ and $t=30$.
\end{itemize}
\end{theorem}

\begin{theorem} \label{overall4}
For odd $n \geq 33$, let $I$, $I'$, $d$ and $d'$ be the unique integers such that $n= 18I +9 + 2d$ and $n = 22I'+11+2d'$, $I,I' \geq 1$, $0 \leq d < 9$ and $0 \leq d' < 11$. Then there exist four transversals of $B_n$ that intersect stably in $t$ points for $t\in [min(3+d',d),  n] \setminus (\{n-15\} \cup [n-7, n-1])$, except, perhaps, when:

\begin{itemize}
\item $33 \leq n \leq 43$ and $t \in [10+d',11+d'] \cup [ n-14, n-12]$,
\item $45 \leq n \leq 53$ and $t \in [-1+d', 2+d'] \cup [10+d',11+d'] \cup [18+d', 20+d']$,
\item $63 \leq n \leq 75$ and  $t \in [7+d, 8+d]$.
\end{itemize}
\end{theorem}

\section{Results}

\subsection{Basic results}

\begin{lemma}
For an odd integer $n$, there exists a set of $\mu$ transversals of $B_n$ which intersect stably in $n$ points, for any $\mu \geq1$.
\end{lemma}
\begin{proof}
For odd $n$, the main diagonal's cells $(i,i,2i)$ with $i \in [0,n-1]$ form a transversal of $B_n$, showing at least one transversal exists. 
A set of $\mu$ transversals with each transversal identical intersects stably in $n$ points.
\end{proof}

\begin{lemma}
For an odd integer $n$, there exists a set of $\mu$ transversals of $B_n$ which intersect stably in $0$ points, for any $1 \leq \mu  \leq n$.
\end{lemma}
\begin{proof}
Consider the $\mu$ transversals of $B_n$ given by $T_{\alpha} = \{(i, i+\alpha, 2i+\alpha \mid i \in [0,n-1]\}$ for $\alpha \in [0,\mu-1]$. The $\mu$ transversals intersect stably in $0$ points.
\end{proof}

\begin{lemma} \label{lem:wanlessIdea}
For odd integer $n$ with $n=m_1 m_2$, there exists a set of $m_2$ transversals of $B_n$ which intersect stably in $n-m_2$ points.
\end{lemma}
\begin{proof}
Let $T^1_{\alpha} = \{(m_1 i, m_1 (i+\alpha),  m_1 (2i+\alpha)) \mid 0 \leq i \leq  m_2-1\}$ and $T^2 = \{ i, i,  2i  \mid 0 < i \leq  n-1 \text{ and } m_1 \nmid i\}$, for $\alpha \in [1,m_2]$. 
Define transversals $T_{\alpha} = T^1_\alpha \cup T^2$ for $\alpha \in [1,m_2]$.  
Then $T_{\alpha} \cap T_{\beta} = T^2$, for each $1 \leq \alpha<\beta \leq m_2$, where $|T^2|=n-m_2$. 
\end{proof}

\begin{lemma} \label{lem:nonexistence}
For an odd integer $n$, there does not  exists a set of $\mu$ transversals of $B_n$ that intersect stably in $t$ points, for $t \in \{n-\mu+1, \ldots, n-1\}$, for any $\mu \geq 2$.
\end{lemma}
\begin{proof}
Suppose that there exists a set of $\mu$ transversals that intersect stably in $t$ points, for $t \geq 1$.
Let $C\subseteq [0,n-1]$ be the set of columns such that no pair of transversals of our set of $\mu$ transversal intersect in column $c \in C$.
If row $r$ has no pair of transversals intersect in row $r$, then the set $\{(r,c',e' ) \in T_{\alpha} \mid 1 \leq \alpha \leq \mu \text{ and } c' \in C\}$ has size $\mu$. 
But this implies $|C|\geq \mu$, which means there can be at most $n-\mu$ columns where the $\mu$ transversals meet. 
This implies the result.
\end{proof}

%
%

\subsection{Computer search}

We performed a computer search for $\mu$ transversals of $B_n$ when $n$ is relatively small, and $\mu=3,4$.

For $n \in \{5,7,9,11,13,15\}$, the program was able to exhaustively check the search space. For $n \in \{17,19,21,23,25,27,29,31\}$, we were only able to obtain partial results, as the search space was quite large. The results are summarized in Tables \ref{tab:tab3-trans} and \ref{tab:tab4-trans}.

\begin{table}[!h]
\begin{center}
  \begin{tabular}{ |c | c | c|  }
    \hline
$n$& $t$\\ \hline

5& 1
\\ \hline
7& 1,2
\\ \hline
9& 1,2,3,4,5,6
\\ \hline
11& 1,2,3,4,5,6,7
\\ \hline
13& 1,2,3,4,5,6,7,9
\\ \hline
 \hline

15& 1,2,3,4,5,6,7,8,9,10,11,12
\\ \hline
17& 1,2,3,4,5,6,7,8,9,10,11
\\ \hline
19& 1,2,3,4,5,6,7,8,9,10,11,12,13
\\ \hline
21& 3,4,5,6,7,8,9,10,11,12,13,14,15,16,18
\\ \hline
23& 4,5,6,7,8,9,10,11,12,13,14,15,16
\\ \hline
25& 6,7,8,9,10,11,12,13,14,15,16,17,18,21
\\ \hline
27& 7,8,9,10,11,12,13,14,15,16,17,18,19,20,21,22,24
\\ \hline
29& 9,10,11,12,13,14,15,16,17,18,19,20,21,22                                                                                                                                                                                                                                                                                                                                                                                                                                                                                                                                                                                                                                                                                                                                                                                                                                                                                                                                                                                                                                                                                                                                                                                                                                                                                                                                                                                                                                                                                                                                                                                                                                                                                                                                                                                                                                                                                                                                                                                                                                                                                                                                                                                                                                                                                                                                                                                                                                                                                                                                                                                                                                                                                                                                                                                                                                                                                                                                                                                                                                                                                                                                                                                                                                                                                                                                                                                                                                                                                                                                                                                                                                                                                                                                                                                                                                                                                                                                                                                                                                                                                                                                                                                                                                                                                                                                        
\\ \hline
31& 13,14,15,16,17,18,20
\\ \hline

\end{tabular}
\end{center}
\caption{There exists a set of $\mu=3$ transversals of $B_n$ with stable intersection size $t$.} \label{tab:tab3-trans}
\end{table}

\begin{table}[!h]
\begin{center}
  \begin{tabular}{ |c | c | c|  }
    \hline
$n$& $t$\\ \hline

7& 1
\\ \hline
9& 1,2
\\ \hline
11& 1,2,3
\\ \hline
13& 1,2,3,4
\\ \hline
\hline

15& 1,2,3,4,5,10
\\ \hline
17& 1,2,3,4,5,6
\\ \hline
19& 1,2,3,4,5,6,7
\\ \hline
21& 3,4,5,6,7,8
\\ \hline
23& 4,5,6,7,8,9
\\ \hline
25& 6,7,8,9,10
\\ \hline
27& 7,8,9,10,11
\\ \hline
29& 9,10,11,12
\\ \hline
31& 13
\\ \hline
\end{tabular}
\end{center}
\caption{There exists a set of $\mu=4$ transversals of $B_n$ with stable intersection size $t$.}  \label{tab:tab4-trans}
\end{table}

\subsection{Principal construction}

For this section, take $n$ to be a fixed odd integer. 
Define $B^n_{i,j}$ to be the $j \times j$ subsquare of $B_n$ at the intersection of rows and columns with indexes $i, i+1, \ldots, i+j-1$. 
We will write  $B_{i,j}$ instead of  $B^n_{i,j}$ when the value of $n$ is clear in the given context. 
The cells of such a subsquare are filled with symbols from $\{2 i, \ldots, 2  i+2  j-2\}$. 
We consider a partial transversal within the $j \times j$ subsquare of cells $B_{i,j}$ to be a set of $j$ triples $\{(r_k,c_k,e_k) \mid r_k-i,c_k-i \in [0,j-1] \text{ and } 0 \leq k \leq j-1\}$ such that $| \{r_k \mid 0 \leq k \leq j-1\}|=|\{c_k \mid 0 \leq k \leq j-1\}|=|\{e_k \mid 0 \leq k \leq j-1\}| = j$. 
We further consider a set of subsquares $B_{i,j}$ such that the subsquares  partition the rows and columns of $B_n$. 
Then finding certain sets of $\mu$ partial transversals within each of these $B_{i,j}$ that only use certain symbols will amount to a set of $\mu$ transversals of $B_n$.

Fix $\mu$ an integer and $b$ an odd integer with $3 \leq b \leq n/3$. 
For a given $\bar{d}$ and $t$, we write $t \in \Omega^{b}_{\mu}(b+\bar{d})$ if there exists a set of $\mu$ partial transversals of $B_{0,b+\bar{d}}$ within $B_n$ that intersect stably in $t$ points, and only use symbols from $\{({b-1})/{2}, \ldots, 3(b-1)/2+2\bar{d}\} \setminus  \{b+2j \mid 0 \leq j < \bar{d}\}$. 
Notice that $B_{i,j}$ is just a relabeling of the symbols of $B_{0,j}$, and so the existence of a set of $\mu$ partial transversals within $B_{0,j}$ is equivalent to the existence of a set of $\mu$ partial transversals within $B_{i,j}$.

Take $I$ and $d$ to be the unique integers with $I \geq 1$, $d \in [0,b-1]$, and $n = 2Ib+b+2d$.  
We consider three types of subsquares; large subsquares $B_{0,b+d}$, small subsquares $B_{(n+b)/2,d}$ and base subsquares $B_{ib+d,b}$ and $B_{(I+i)b+2d,b}$ for $1 \leq i \leq I$.
Figure \ref {fig:layout} shows the layout of the subsquares. 
The symbols that fill the cells of the partial transversal from each of these subsquares are restricted. 
In particular, the base subsquares $B_{ib+d,b}$ use the symbols $\{2(ib+d) + (b-1)/2, \ldots, 2(ib+d)+3(b-1)/2  \}$ and 
the base subsquares $B_{(I+i)b+2d,b}$ use the symbols $\{2(Ib+ib+2d)+(b-1)/2, \ldots, 2(Ib+ib+2d)+3(b-1)/2  \}$ for $1 \leq i \leq I$, 
the large subsquare $B_{0,b+d}$ use the symbols $\{({b-1})/{2}, \ldots, 3(b-1)/2 +2d\} \setminus  \{b+2j \mid 0 \leq j < d\}$,  
 and the small subsquare $B_{(n+b)/2,d}$ use the symbols $\{b+2j \mid 0 \leq j < d\}$.

These symbols have been chosen so that the partial transversal of one of the subsquares does not share any symbols in common with the partial transversal of any other subsquare. 
We demonstrate the interleaving that occurs for the base subsquares in Figure \ref{fig:interleave}.

\begin{figure}
\begin{center}
 \scalebox{.7}{
\Large
\begin{tikzpicture}[scale=0.86]

\draw [decorate,decoration={brace,amplitude=10pt},xshift=0pt,yshift=0pt]
 (0,14.1) -- (3,14.1)  node [black,midway,xshift=-0.6cm]{};
\node at (1.5,15) {b+d};

\draw [decorate,decoration={brace,amplitude=10pt},xshift=0pt,yshift=0pt]
 (3,14.1) -- (5,14.1)  node [black,midway,xshift=-0.6cm]{};
\node at (4,15) {b};

\draw [decorate,decoration={brace,amplitude=10pt},xshift=0pt,yshift=0pt]
 (6,14.1) -- (8,14.1)  node [black,midway,xshift=-0.6cm]{};
\node at (7,15) {b};

\draw [decorate,decoration={brace,amplitude=6pt},xshift=0pt,yshift=0pt]
 (8,14.1) -- (9,14.1)  node [black,midway,xshift=-0.6cm]{};
\node at (8.5,15) {d};

\draw [decorate,decoration={brace,amplitude=10pt},xshift=0pt,yshift=0pt]
 (9,14.1) -- (11,14.1)  node [black,midway,xshift=-0.6cm]{};
\node at (10,15) {b};

\draw [decorate,decoration={brace,amplitude=10pt},xshift=0pt,yshift=0pt]
 (12,14.1) -- (14,14.1)  node [black,midway,xshift=-0.6cm]{};
\node at (13,15) {b};


\draw (0,0) --(0,14) -- (14,14) -- (14,0) -- (0,0);
\draw (0,14) --(0,11) -- (3,11) -- (3,14) -- (0,14);
\node at (1.5,12.5) {large};
\draw (3,11) --(3,9) -- (5,9) -- (5,11) -- (3,11);
\node at (4,10) {base};
\node at (4,9.5) {$_1$};
\draw [densely dashed] (5.1,8.9) -- (5.9,8.1);
\draw (6,8) --(6,6) -- (8,6) -- (8,8) -- (6,8);
\node at (7,7) {base};
\node at (7,6.5) {$_I$};
\draw (9,5) --(9,6) -- (8,6) -- (8,5) -- (9,5);
\node at (10,7) {small};
\draw (10,6.7) -- (8.5,5.5);
\draw [densely dashed] (11.1,2.9) -- (11.9,2.1);
\draw (9,5) --(9,3) -- (11,3) -- (11,5) -- (9,5);
\node at (10,4) {base};
\node at (10,3.5) {$_{I+1}$};
\draw (12,2) --(12,0) -- (14,0) -- (14,2) -- (12,2);
\node at (13,1) {base};
\node at (13,0.5) {$_{2I}$};
\end{tikzpicture}
}
\caption{The positioning of the subsquares. By taking the union of partial transversals of $B_n$ in these subsquares, we find a transversal of $B_n$.} \label{fig:layout}
\end{center}

\end{figure}
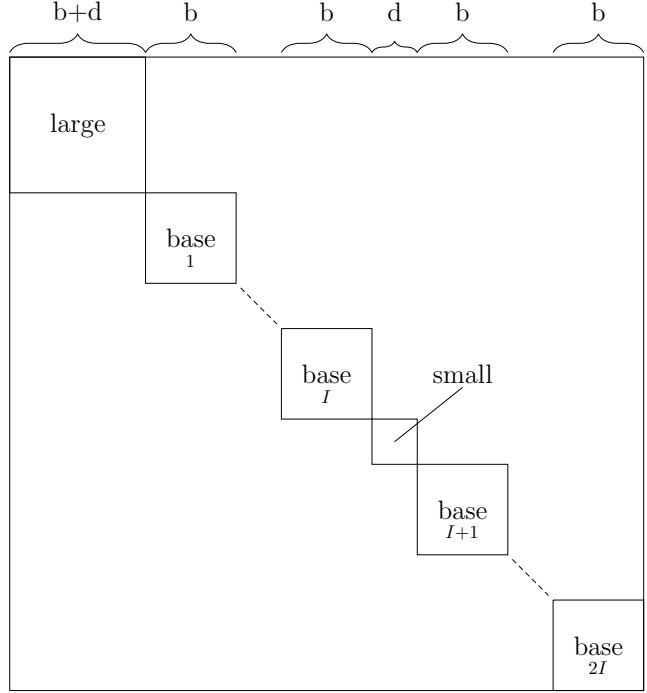

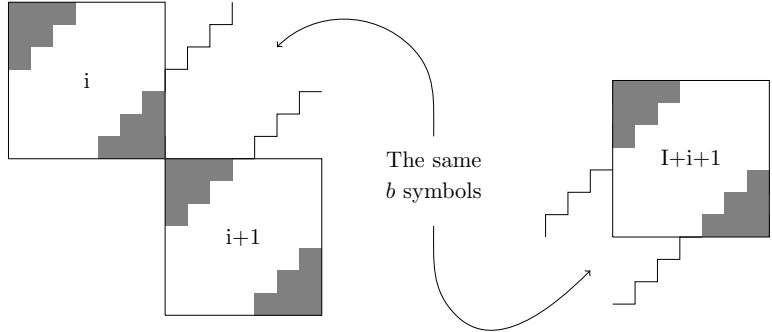
\begin{figure}
\begin{center}

 \scalebox{.85}{

\begin{tikzpicture}[scale=0.35]

\path [fill=gray] (4,0) rectangle (7,1);
\path [fill=gray] (5,0.9) rectangle (7,2);
\path [fill=gray] (6,1.9) rectangle (7,3);

\path [fill=gray] (0,7) rectangle (3,6);
\path [fill=gray] (0,6.1) rectangle (2,5);
\path [fill=gray] (0,5.1) rectangle (1,4);

\draw (0,0) --(0,7) -- (7,7) -- (7,0) -- (0,0);
\node at (3.5,3.5) {i+1};

\path [fill=gray] (-3,7) rectangle (0,8);
\path [fill=gray] (-2,7.9) rectangle (0,9);
\path [fill=gray] (-1,8.9) rectangle (0,10);

\path [fill=gray] (-7,14) rectangle (-4,13);
\path [fill=gray] (-7,13.1) rectangle (-5,12);
\path [fill=gray] (-7,12.1) rectangle (-6,11);

\draw (-7,7) --(-7,14) -- (0,14) -- (0,7) -- (-7,7);
\node at (-3.5,10.5) {i};

\draw (0,10) -- (0,11) -- (1,11) -- (1,12) -- (2,12) -- (2,13) -- (3,13) -- (3,14);
\draw (3,7) -- (4,7) -- (4,8) -- (5,8) -- (5,9) -- (6,9) -- (6,10) -- (7,10);

\draw[->] (12,4) to [out=-90,in=-45+180]  
                 (13,0) to [out=135+180,in=45+180]  
	      (19,2) ;
\draw[->] (12,8) to [out=90,in=-45]  
                 (11,12) to [out=135,in=45]  
	      (5,12) ;

\node at (12,7) {The same};
\node at (12,5.5) {$b$ symbols};

\path [fill=gray] (24,3.5) rectangle (27,4.5);
\path [fill=gray] (25,4.4) rectangle (27,5.5);
\path [fill=gray] (26,5.4) rectangle (27,6.5);

\path [fill=gray] (20,10.5) rectangle (23,9.5);
\path [fill=gray] (20,9.6) rectangle (22,8.5);
\path [fill=gray] (20,8.6) rectangle (21,7.5);

\draw (20,3.5) --(20,10.5) -- (27,10.5) -- (27,3.5) -- (20,3.5);
\node at (23.5,7) {I+i+1};

\draw  (20,6.5)--(19,6.5) -- (19,5.5)--(18,5.5)--(18,4.5)--(17,4.5) --(17,3.5);
\draw (24,3.5)--(23,3.5)--(23,2.5)--(22,2.5)--(22,1.5)--(21,1.5)--(21,0.5) -- (20,0.5);

\end{tikzpicture}

}

\end{center}
\caption{We choose partial transversal such that the symbols not used between the $i$th and $(i+1)$th base subsquares are used in the $(I+i+1)$th base subsquare. The darkened cells represent those cells that we do not allow in any partial transversal.}  \label{fig:interleave}
\end{figure}

\begin{theorem} \label{mainTheorem} 
Let $n,b$ be odd integers, $3 \leq b \leq n/3$, and $\mu$ an integer with $\mu \geq 2$. 
 Let $I \geq 1$ and $d \in [0,b-1]$ be the unique integers such that $n=2 b I + b+2d$. 
There exists $\mu$ transversals of $B_n$ that intersect stably in $t$ points with $t=d+ \sum_{i=0}^{2I} t_i$, where $t_0 \in {\Omega}^{b}_{\mu}(b+d)$ and $t_i \in {\Omega}^{b}_{\mu}(b)$, for $1 \leq i \leq 2I$.
\end{theorem}

We provide the following construction, followed by a proof that demonstrates that the construction yields Theorem \ref{mainTheorem}.

\begin{construction}
\label{mainConstruction}
Take $\mu \geq 2$ an integer and $n,b$ odd integers, $3 \leq b \leq n/3$
 Let $I \geq 1$ and $d \in [0,b-1]$ be the unique integers such that $n=2 b I + b+2d$. 

We will construct $\mu$ subsets of $B_n$, $T_{1}, \ldots, T_{\mu}$ by finding partial transversals selected from a large subsquare $B_{0,b+d}$, a small subsquare $B_{b(I+1)+d,d} = B_{({n+b})/{2},d}$, and base subsquares $B_{bi+d,b}$ and $B_{b(I+i)+2d,b}$ for $1 \leq i \leq I$.

For the large subsquare, as $t_0 \in {\Omega}^{b}_{\mu}(b+d)$ there exists a set of $\mu$ partial transversals $P_1^L, \ldots, P_{\mu}^L$ within $B_{0,b+d}$ that intersects stably in $t_0$ points and using each symbol of $\{(b-1)/2, \ldots, 2d +3(b-1)/2\} \setminus \{b+2 d' \mid 0 \leq d' < d\}$ precisely once per partial transversal. 
We place the cells of $P_{\beta}^L$ into $T_{\beta}$, $1 \leq {\beta} \leq \mu$.

For the small subsquare, define a set of $\mu$ partial transversals $P_1^S, \ldots, P_{\mu}^S$ within $B_{b(I+1)+d,d}$ that intersect stably in $d$ points by placing cells $(r,r,2r)$ with $r=b(I+1)+d+d'$ into every partial transversal $P_\beta^S$, $1\leq \beta \leq \mu$ and $0 \leq d' <d$, so that each of the $\mu$ partial transversals are identical. We place the cells of $P_\beta^S$ into $T_\beta$, $1\leq \beta \leq \mu$.

For the first set of base subsquares, $B_{bi+d,b}$ with $1 \leq i \leq I$, as $t_i \in {\Omega}^{b}_{\mu}(b)$ there exists a set of $\mu$ partial transversals $P_1^i, \ldots, P_\mu^i$ of $B_{0,b}$ that intersect stably in $t_i$ points and using each symbol of $\{(b-1)/2, \ldots, 3(b-1)/2\}$ precisely once. For every $(r,c,e) \in P_{\beta}^i$, place the cells $(r+a,c+a,e+2a)$ with $a=bi+d$  into $T_{\beta}$, $1 \leq {\beta} \leq \mu$. The cells that were just filled are in the subsquare $B_{bi+d,b}$.

For the second set of base subsquares, $B_{b(I+i)+2d,b}$ with $1 \leq i \leq I$, as $t_{I+i} \in {\Omega}^{b}_{\mu}(b)$ there exists a set of $\mu$ partial transversals $P_1^{I+i}, \ldots, P_\mu^{I+i}$ of $B_{0,b}$ that intersect stably in $t_{I+i}$ points and using each symbol of $\{(b-1)/2, \ldots, 3(b-1)/2\}$ precisely once. For every $(r,c,e) \in P_{\beta}^{I+i}$, place the cells $(r+a,c+a,e+2a)$ with $a=b(I+i)+2d$  into  $T_{\beta}$, $1 \leq {\beta} \leq \mu$. The cells that were just filled are in the subsquare $B_{b(I+i)+2d,b}$.

\end{construction}

\begin{proof}
We begin by showing that $T_1, \ldots, T_\mu$ from Construction \ref{mainConstruction} are each diagonals.  
Consider any  $T \in \{T_1, \ldots, T_\mu\}$. 
As $T$ is the union of partial transversals of subsquares, each of which share no common row or common, clearly $T$ is a selection of $n$ cells of $L$ using each row (resp.\ column) once, and so $T$ is a diagonal of $B_n$.

We will proceed to show that each diagonal $T\in \{T_1, \ldots, T_\mu\}$ is a transversal, and that they intersect stably in $d+ \sum_{i=0}^{2I} t_i$ points.
The construction placed $2b$ filled cells from the two base subsquares $B_{bi+d,b}$ and  $B_{b(I+i)+2d,b}$ into $T$, for each fixed $i$, $1 \leq i \leq I$.
This consisted of precisely one filled cell for each symbol of $ \{2bi+2d -(b+1)/2, \ldots, 2bi+2d+3(b-1)/2  \}$. 

Then collectively the $2I$ base subsquares were used to fill $2bI$ cells into $T$, placing precisely one filled cell for each symbol of $ \{2d +3(b-1)/2+1, \ldots, 2bI+2d+3(b-1)/2  \}$.

During the construction, $(b+d)+d$ filled cells were placed into $T$ from the large subsquare $B_{0,b+d}$ and the small subsquare $B_{b(I+1)+d,d}$, which had one filled cell for each symbol of $ \{(b-1)/2, \ldots, 2d+3(b-1)/2\}$. 

Combining the statements for the $2  I$ base subsquares and the large and small subsquare, each symbol of $\{(b-1)/2, \ldots, 2bI+2d+3(b-1)/2  \}=\{0, \ldots, 2bI + b +2d-1\}$ appears in the diagonal $T$ precisely once, after recalling that each symbol is taken modulo $2bI  + b+ 2d$ and noting that $2bI+2d +3(b-1)/2 = (2bI + b + 2d) + (b-1)/2-1 \equiv ({b-1})/{2} -1 \pmod{n}$. 

This shows that $T$ is indeed a transversal, and so the construction has indeed formed $\mu$ transversals. 
Now we need to show that the $\mu$ transversals intersect stably in $d+ \sum_{i=0}^{2I} t_i$ points. 
Suppose the $\mu$ partial transversals we chose for the large subsquare intersect stably in the set $S_0$, the $\mu$ partial transversals we chose for the base subsquare  $B_{bi+d,b}$  intersect stably in the set $S_i$, and the $\mu$ partial transversals we chose for the base subsquare $B_{b(I+i)+2d,b}$ intersect stably in the set $S_{I+i}$, for $1 \leq i \leq I$. 
Clearly the $\mu$ partial transversals we chose for the small subsquare intersect stably in the points $S_{-1} = \{(r,r,2r) \mid r=b(I+1)+d+d' \text{ and } 1 \leq d' \leq d \}$. 
The size of $S_i$ is $|S_i|=t_i$, for $0 \leq i \leq 2I$, and $|S_{-1}|=d$. 
The  $\mu$ transversals then clearly intersect stably in the $d+\sum_{i=0}^{2I}t_i$ points  $\bigcup_{i=-1}^{2I} S_i$. 
\end{proof}

\begin{example}
 We consider the case when $\mu=2$, $n=17$, $b=5$, $d=1$, $I=1$, $t_0=1$, $t_1=1$, $t_2=0$. Note that $n=2b I + b+2d=5\cdot 2 \cdot 1 + 5 + 2\cdot 1=17$.

For this example, we will represent the first transversal of $B_{i,j}$ or $B_n$ by underlining those entries, and the second transversal by adding a superscripted star. The intersection of the two (partial) transversals are those entries that are both underlined and starred. 

For the small subsquare, we require $\mu=2$ transversals of the small subsquare $B_{(I+1)b+d,d}=B_{11,1}$ that intersect stably in $d=1$ points. These transversals are simply chosen as there is only one cell in $B_{11,1}$. 

For the large subsquare, we require $\mu=2$ transversals of the large subsquare $B_{0,b+d}=B_{0,6}$ that intersect stably in $t_0=1$ points using symbols $\{(b-1)/2, \ldots, 2d +3(b-1)/2\} \setminus \{b+2 d' \mid 0 \leq d' < d\}   =   \{2, 3,4,6,7,8\}$, for example:

 \begin{center}
  \begin{tabular}{ |c | c | c| c | c | c|  }
    \hline
$0$&$1$&$2$&$3^*$&$\underline{4}$&$5$ \\
          \hline
$1$&$2$&$\underline{3}$&$4$&$5$&$6^*$ \\
          \hline
${\underline{2}}^*$&$3$&$4$&$5$&$6$&$7$ \\
          \hline
$3$&$4^*$&$5$&$6$&$7$&$\underline{8}$ \\
          \hline
$4$&$5$&$6$&$\underline{7}$&$8^*$&$9$ \\
          \hline
$5$&$\underline{6}$&$7^*$&$8$&$9$&$10$ \\
          \hline
  \end{tabular}
 \end{center}

For the first set of base subsquares (in this case the set contains only one subsquare) we require $\mu=2$ transversals of the base subsquare $B_{0,b}=B_{0,5}$ that intersect stably in  $t_1=1$ points and using symbols $\{(b-1)/2, \ldots, 3(b-1)/2\} =   \{2, 3,4,5,6\}$, for example:

 \begin{center}
  \begin{tabular}{ |c | c | c| c | c | c|  }
    \hline
$0$&$1$&$2$&$3^*$&$\underline{4}$\\
          \hline
$1$&$2$&$\underline{3}$&$4$&$5^*$\\
          \hline
${\underline{2}^*}$&$3$&$4$&$5$&$6$\\
          \hline
$3$&$4^*$&$5$&$\underline{6}$&$7$\\
          \hline
$4$&$\underline{5}$&$6^*$&$7$&$8$\\
          \hline
  \end{tabular}
\end{center}

For the second set of base subsquares (in this case the set contains only one subsquare), we require $\mu=2$ transversals of the base subsquare $B_{0,b}=B_{0,5}$ that intersect stably in  $t_2=0$ points, for example:

 \begin{center}
  \begin{tabular}{ |c | c | c| c | c | c|  }
    \hline
$0$&$1$&$2$&$3^*$&$\underline{4}$\\
          \hline
$1$&$2^*$&$\underline{3}$&$4$&$5$\\
          \hline
${\underline{2}}$&$3$&$4$&$5$&$6^*$\\
          \hline
$3$&$4$&$5^*$&$\underline{6}$&$7$\\
          \hline
$4^*$&$\underline{5}$&$6$&$7$&$8$\\
          \hline
  \end{tabular}
\end{center}

Then we can obtain $\mu=2$ transversals of size $n=17$ and stable intersection size $d+t_0 + t_1 + t_2 =3$ as (where we omit those entries not relevant to our construction):

\scalebox{0.75}{
  \begin{tabular}{ c |c | c | c| c | c | c|c | c | c| c | c | c|c | c | c| c | c |   }
  \multicolumn{1}{c}{}& \multicolumn{1}{c}{0}&\multicolumn{1}{c}{1}& \multicolumn{1}{c}{2}& \multicolumn{1}{c}{3}& \multicolumn{1}{c}{4}& \multicolumn{1}{c}{5}& \multicolumn{1}{c}{6}& \multicolumn{1}{c}{7}& \multicolumn{1}{c}{8}& \multicolumn{1}{c}{9}& \multicolumn{1}{c}{10}& \multicolumn{1}{c}{11}& \multicolumn{1}{c}{12}& \multicolumn{1}{c}{13}& \multicolumn{1}{c}{14}& \multicolumn{1}{c}{15}& \multicolumn{1}{c}{16}\\
    \cline{2-18}
0&  $0$  &  $1$  &  $2$  &  $3^*$  &  $\underline{4}$  &  $5$ &&&&&&&&&&&\\
          \cline{2-18}
1&  $1$  &  $2$  &  $\underline{3}$  &  $4$  &  $5$  &  $6^*$ &&&&&&&&&&& \\
          \cline{2-18}
2&  ${\underline{2^*}}$  &  $3$  &  $4$  &  $5$  &  $6$  &  $7$ &&&&&&&&&&& \\
          \cline{2-18}
3&  $3$  &  $4^*$  &  $5$  &  $6$  &  $7$  &  $\underline{8} $  &&&&&&&&&&& \\
          \cline{2-18}
4&  $4$  &  $5$  &  $6$  &  $\underline{7}$  &  $8^*$  &  $9 $  &&&&&&&&&&& \\
          \cline{2-18}
5&  $5$  &  $\underline{6}$  &  $7^*$  &  $8$  &  $9$  &  $10 $  &&&&&&&&&&& \\
          \cline{2-18}

6&&&&&&&  $12$  &  $13$  &  $14$  &  $15^*$  &  $\underline{16}$  &&&&&&\\
          \cline{2-18}
7&&&&&&&  $13$  &  $14$  &  $\underline{15}$  &  $16$  &  $0^*$  &&&&&&\\
          \cline{2-18}
8&&&&&&&  ${\underline{14^*}}$  &  $15$  &  $16$  &  $0$  &  $1$  &&&&&&\\
          \cline{2-18}
9&&&&&&&  $15$  &  $16^*$  &  $0$  &  $\underline{1}$  &  $2$  &&&&&&\\
          \cline{2-18}
10&&&&&&&  $16$  &  $\underline{0}$  &  $1^*$  &  $2$  &  $3$  &&&&&&\\
          \cline{2-18}

11&&&&&&&&&&&&  ${\underline{5^*}}$  &&&&&\\

         \cline{2-18}

12&&&&&&&&&&&&&  $7$  &  $8$  &  $9$  &  $10^*$  &  $\underline{11}$\\
          \cline{2-18}
13&&&&&&&&&&&&&  $8$  &  $9^*$  &  $\underline{10}$  &  $11$  &  $12$\\
          \cline{2-18}
14&&&&&&&&&&&&&  $\underline{9}$  &  $10$  &  $11$  &  $12$  &  $13^*$\\
          \cline{2-18}
15&&&&&&&&&&&&&  $10$  &  $11$  &  $12^*$  &  $\underline{13}$  &  $14$\\
          \cline{2-18}
16&&&&&&&&&&&&&  $11^*$  &  $\underline{12}$  &  $13$  &  $14$  &  $15$\\
         \cline{2-18}
  \end{tabular}
}

\

This concludes the example.
\end{example}

\section{Application to $\mu=3,4$}

Our approach to finding $3$ (resp.\ $4$) transversals of $B_n$ that intersect stably is to find $3$ (resp.\ $4$) partial transversals of $B_{i,j}$ that intersect stably for certain values of $i$ and $j$, and compose these into transversals of $B_n$.
 We will use Theorem \ref{mainTheorem} with base sizes of $b=9,11,15$.
These base sizes have been chosen based upon the results of a computational search for partial transversals of base and large subsquares. 

The appendix includes tables that contains a set of either three or four rows, corresponding to $\mu=3$ and $\mu=4$ respectively, each row containing $b+d$ symbols. 
For the $r$th row, denote the $i$th symbol of the list in this row as $a^r_i$, for $1 \leq r \leq \mu$ and $1 \leq i \leq b+d$.  
The cells $(i,a^r_i)$, $1 \leq i \leq b+d$, form a partial transversal of $B_{0,b+d}$. 
The three (resp.\ four) rows give three (resp.\ four) partial transversals of $B_{0,b+d}$, each partial transversal having $t$ cells that are common amongst all three (resp.\ four) partial transversals, and $b+d-t$ cells which do not appear in the other partial transversals. 
We call this representation \emph{reduced form}.

We take the addition and scalar multiplication of finite sets to be:

\[
A+B = \{a+b \mid a \in A, b \in B\}
\]
\[
kA = \{\sum_{i=1}^k a_i \mid a_i \in A\}
\]

\begin{lemma}
\label{lemmaSet}
Let $j,a,b$ be positive integers with $1 \leq a<b$. We have
$j ( [0,a]\cup \{b\}) = [0,j b] \setminus \bigcup_{i=1}^{\lfloor ({b-2})/{a}\rfloor} [jb-ib +i a+1, jb-ib+b  -1]$.
\end{lemma}
\begin{proof}
From definition, $j ([0,a] \cup \{b\}) = \{\sum_{i=1}^j a_i \mid a_i \in [0,a] \cup \{b\}\} = \bigcup_{i=0}^{j} [i \cdot b, i \cdot b + (j-i)a]= \bigcup_{i=0}^{j} [(j-i)b, (j-i)b + i a]$. 
Then any value $t\in [0,j b]$ with $t \notin j ([0,a] \cup \{b\})$ must be between the two intervals $[(j-i')b, (j-i')b + i' a]$ and $[(j-i'+1)b, (j-i'+1)b + (i'-1) a]$ for some $1 \leq i' \leq j$, and hence $t \in [(j-i') b + i' a+1, (j-i'+1) b-1]$. 
This proves the result, once we note that $[(j-i')b + i' a+1, (j-i'+1) b-1]$ is non-empty only when $i'a+1 \leq b-1$, and so $i' \leq (b-2)/a$.
%
\end{proof}

\subsection{Existence of partial transversals in subsquares}

It is important to note that $ {\Omega}_4^{b}(b+d) \subseteq {\Omega}_3^{b}(b+d)$. Also, if there is at least one partial transversal of $B_{0,b+d}$ using symbols $\{(b-1)/2, \ldots, 2d +3(b-1)/2\} \setminus \{b+2 d' \mid 0 \leq d' < d\}$ then $b+d \in{\Omega}_{\mu}^{b}(b+d)$ for any $\mu \geq 2$. 
This also tells us that if ${\Omega}_{\mu}^{b}(b+d) \neq \emptyset$, then $b+d \in{\Omega}_{\mu'}^{b}(b+d)$ for each $\mu' \geq \mu$.

\begin{lemma}
\label{baseblocks}
The following hold\footnote{In each case equality holds, but this strengthened statement is not needed.}:
\begin{enumerate}
\item ${\Omega}_4^{9}(9) \supseteq \{0,1,9\}$.
\item ${\Omega}_4^{11}(11) \supseteq \{0,1,2,3,11\}$.
\item ${\Omega}_4^{15}(15) \supseteq \{1,2,3,4,5,15\}$.
\item ${\Omega}_3^{9}(9) \supseteq  \{0,1,2,3,9\}$.
\item ${\Omega}_3^{11}(11) \supseteq \{0,1,2,3,4,5,11\}$.
\end{enumerate}
\end{lemma}
\begin{proof}
The corresponding partial transversals have been found by a computer search, and have been written in reduced form in the appendix, in respectively 
 Table \ref{tab:tab4-9a}, 
Table \ref{tab:tab4-11a}, 
Table \ref{tab:tab4-15a}, 
Tables \ref{tab:tab4-9a} and \ref{tab:tab3-9a}, and 
Tables \ref{tab:tab4-11a} and \ref{tab:tab3-11a}.  
\end{proof}

\begin{lemma}
\label{low}
The following  hold:
\begin{enumerate}
\item $0 \in {\Omega}_4^{9}(9+d)$, for all $0 \leq d <9$.
\item $3 \in {\Omega}_4^{11}(11+d)$, for all $0 \leq d <11$.
\end{enumerate}
\end{lemma}
\begin{proof}
The corresponding partial transversals have been found by a computer search, and have been written in reduced form in the appendix, in respectively Tables \ref{tab:tab4-9a} and \ref{tab:tab4-9b} , and Tables  \ref{tab:tab4-11a} and \ref{tab:tab4-11b}.
\end{proof}

\begin{lemma}
\label{high}
The following  hold:
\begin{enumerate}
\item $11+d \in {\Omega}_4^{11}(11+d)$, for all $0 \leq d <11$.
\item $15+d \in {\Omega}_4^{15}(15+d)$, for all $0 \leq d <15$.
\end{enumerate}
\end{lemma}
\begin{proof}
Since the partial transversals required  intersect stably in the same number of points as the square size, we only need one partial transversal of $B_{0,b+d}$, which is repeated $4$ times to form the $4$ partial transversals that intersect stably in $b+d$ points. One partial transversal has been been found for each of the $26$ cases by a computer search, and these have been written in reduced form in the appendix, in respectively Tables \ref{tab:tab4-11a} and \ref{tab:tab4-11b}, and Table \ref{tab:tab4-15b}.
\end{proof}

\begin{lemma}
\label{lemma2I}
The following set relations hold:
\begin{enumerate}
\item $2I  {\Omega}_4^{9}(9) = [0, 18I] \setminus \cup_{i=1}^7 [18I-8i+1, 18I-9i+8 ]$;
\item $2I {\Omega}_4^{11}(11) = [0, 22I] \setminus (\{22I-23\}\cup [22I-15,  22I-12] \cup [ 22I-7, 22I-1])$;
\item $2I {\Omega}_4^{15}(15) = [2I,  30I] \setminus (\{30I-29\}\cup [30I-19, 30I - 15] \cup [30I - 9,  30I-1])$;
\item $2I  {\Omega}_3^{9}(9) = [0,  18I] \setminus (\{18I-11, 18I-10\} \setminus [ 18I-5,  18I-1 ])$; and 
\item $2I {\Omega}_3^{11}(11) = [0,  22I] \setminus [22I-5,  22I-1] \}$.
\end{enumerate}
\end{lemma}
\begin{proof}
The sets ${\Omega}_{\mu}^{9}(9)$, ${\Omega}_{\mu}^{11}(11)$ and ${\Omega}_4^{15}(15) = \{1\}+\{0,1,2,3,4,14\}$ are given in Lemma \ref{baseblocks} for $\mu=3,4$. 
For the case ${\Omega}_{\mu}^{9}(9)$ and ${\Omega}_{\mu}^{11}(11)$, Lemma \ref{lemmaSet} completes the result for general $J$, however we will only be requiring the case when $J$ is even, and hence written $J=2I$. 
For the case ${\Omega}_4^{15}(15)$, it can be seen that $2I {\Omega}_4^{15}(15) = 2I \{1,2,3,4,5,15\} = 2I (\{1\}+\{0,1,2,3,4,14\}) = \{2I\} + 2I \{0,1,2,3,4,14\}$. 
We can apply Lemma \ref{lemmaSet} to find $2I \{0,1,2,3,4,14\}$, which gives the final result.
\end{proof}

\subsection{$\mu=3$}

\begin{theorem}
\label{3thm-high}
For odd $n \geq 33$, let $I'$ and $d'$ be the unique integers such that $n = 22I'+11+2d'$, $I' \geq 1$ and $0 \leq d' < 11$. Then there exist three transversals of $B_n$ that intersect stably in $t$ points for $t\in [11+2d', n] \setminus [n-5, n-1]$.
\end{theorem}
\begin{proof}
Take the base size to be $b=11$. 
Using Theorem \ref{mainTheorem}, Lemma \ref{high}, and Lemma \ref{lemma2I}, we can conclude there exists the required set of transversals for each $t\in \{d'\}+{\Omega}_3^{11}(11+d') + 2I'  {\Omega}_3^{11}(11) $, and hence for each $t\in \{d'\} + \{11+d'\} + 2I'   {\Omega}_3^{11}(11) =   [11+2d', n] \setminus [n-5, n-1]$.
\end{proof}

\begin{lemma}
\label{3lower18}
For odd $n \geq 27$, let $I$ and $d$ be the unique integers such that $n= 18I +9 + 2d$, $I \geq 1$ and $0 \leq d < 9$. Then there exist three transversals of $B_n$ that intersect stably in $t$ points for $t\in [d, 18I+d] \setminus ([18I-11+d,18I-10+d] \cup [18I-5+d, 18I-1+d])$.

\end{lemma}
\begin{proof}
Take the base size to be $b=9$. 
Using Theorem \ref{mainTheorem}, Lemma \ref{low}, and Lemma \ref{lemma2I}, we can conclude there exists the required set of transversals for each $t\in  \{d\}+\{0\} + 2I   {\Omega}_3^{9}(9)    \subseteq    \{d\}+ {\Omega}_3^{9}(9+d) + 2I {\Omega}_3^{9}(9) $, and hence for each $t\in [d, 18I+d] \setminus ([18I-11+d,18I-10+d] \cup [18I-5+d, 18I-1+d])$.
\end{proof}

\begin{lemma}
\label{3lower22}
For odd $n \geq 33$, let $I'$ and $d'$ be the unique integers such that $n = 22I'+11+2d'$, $I' \geq 1$ and $0 \leq d' < 11$. Then there exist three transversals of $B_n$ that intersect stably in $t$ points for $t\in [3+d', 22I'+3+d'] \setminus [ 22I'-2+d', 22I'+2+d']$.

\end{lemma}
\begin{proof}
Take the base size to be $b=11$. 
Using Theorem \ref{mainTheorem}, Lemma \ref{low}, and Lemma \ref{lemma2I}, we can conclude there exists the required set of transversals for each $t\in  \{d'\}+\{3\} + 2I   {\Omega}_3^{11}(11)   \subseteq    \{d'\}+ {\Omega}_3^{11}(11+d') + 2I   {\Omega}_3^{11}(11) $, and hence for each $t\in[3+d', 22I'+3+d'] \setminus [ 22I'-2+d', 22I+2+d']$.
\end{proof}

\begin{theorem} \label{3thm-low}
For odd $n \geq 33$, let $I$, $I'$, $d$ and $d'$ be the unique integers such that $n= 18I +9 + 2d$ and $n = 22I'+11+2d'$, $I,I' \geq 1$, $0 \leq d < 9$ and $0 \leq d' < 11$. Then there exist three transversals of $B_n$ that intersect stably in $t$ points for $t\in [min(3+d',d), 11+2d']$ except, perhaps, when:
\begin{itemize}
\item $n=51$ and $t=29$,
\item $n=53$ and $t=30$.
\end{itemize}

\end{theorem}
\begin{proof}
We first show that we have three transversals of $B_n$ that intersect stably in $t$ points for $t\in [3+d', 11+2d']$, except in the case $n=51$ and $t=29$, and the case $n=53$ and $t=30,31$.
 Lemma \ref{3lower22} gives the the cases when $t\in [3+d', 22I'-3+d']$. 
Now if $11+2d' \leq 22I'-3+d'$, then we are done. 
Otherwise $d' > 22I'-14$, and since $d' \leq 10$, this implies $I'=1$ and $d' \in \{9,10\}$. 
The case $d'=9$ gives $n=51$, and we do not have  three transversals of $B_n$ that intersect stably in $t$ when $t \in  [22I'-2+d', 11+2d']  = \{29\}$. 
The case $d'=10$ gives $n=53$, and we do not have  three transversals of $B_n$ that intersect stably in $t$ when $t \in  [22I' -2+ d', 11+2d']  = \{30,31\}$. 
We note that the case $n=53$ and $t=31$ is covered by  Lemma \ref{3lower18}.

Second we show that we have those cases with $t\in [d,3+d']$ when $d<3+d'$. For $33 \leq n \leq 43$, $d=3+d'$, so assume $n \geq 45$, implying $I \geq 2$.
By Lemma \ref{3lower18}, we have those cases with  $t\in [d, 18I-12+d]$, and since $18I-12+d \geq 24 > 3+d'$, we are done.
\end{proof}

Then Theorem \ref{overall3} follows by Theorem \ref{3thm-high} and Theorem \ref{3thm-low}.

\subsection{$\mu=4$}

\begin{lemma}
\label{high11}
For odd $n \geq 33$, let  $I'$ and $d'$ be the unique integers such that $n = 22I'+11+2d'$, $I' \geq 1$ and $0 \leq d' < 11$. Then there exist four transversals of $B_n$ that intersect stably in $t$ points for $t\in [11+2d', n] \setminus \{n-23, n-15, \ldots, n-12, n-7, \ldots, n-1\}$.

\end{lemma}
\begin{proof}
Take the base size to be $b=11$. 
Using  Theorem \ref{mainTheorem}, Lemma \ref{high}, and Lemma \ref{lemma2I}, we can conclude there exists the required set of transversals for each $t\in \{d'\}+{\Omega}_4^{11}(11+d') + 2I' {\Omega}_4^{11}(11) $, and hence for each $t\in \{11+2d'\} + 2I'  {\Omega}_4^{11}(11) =  [11+2d', n] \setminus (\{n-23\}\cup [n-15, n-12] \cup [n-7, n-1])$.
\end{proof}

\begin{lemma}
\label{high15}
For odd $n \geq 45$, let $I''$ and $d''$ be the unique integers such that $n = 30I''+15+2d''$, $I'' \geq 1$ and $0 \leq d'' < 15$. Then there exist four transversals of $B_n$ that intersect stably in $t$ points for $t\in [2I'' + 15+2d'', n] \setminus (\{n-29\} \cup [n-19, n-15] \cup [ n-9, n-1])$. 

\end{lemma}
\begin{proof}
Using  Theorem \ref{mainTheorem}, Lemma \ref{high}, and Lemma \ref{lemma2I}, we can conclude there exists the required set of transversals for each $t\in \{d''\}+ {\Omega}_4^{15}(15+d'') + 2I''  {\Omega}_4^{15}(15) $, and hence for each $t\in \{15+2d''\} + 2I''  {\Omega}_4^{15}(15) = [2I'' + 15+2d'', n] \setminus (\{n-29\} \cup [n-19, n-15] \cup [ n-9, n-1])$.
\end{proof}

\begin{theorem} \label{thm-high}
For odd $n \geq 45$, let $I'$ and $d'$ be the unique integers such that $n = 22I'+11+2d'$, $I'  \geq 1$, $0 \leq d' < 11$. Then there exist four transversals of $B_n$ that intersect stably in $t$ points for $t\in [11+2d', n] \setminus (\{n-15\} \cup [n-7, n-1])$. For odd $33 \leq n \leq 43$ such that $n = 33+2d'$ and $0 \leq d' \leq 5$, there exists four transversal of $B_n$ that intersect stably in $t$ points, for $t\in [11+2d', n] \setminus ([n-15, n-12] \cup [n-7, n-1])$.

\end{theorem}
\begin{proof}
Define $I''$ and $d''$ such that $n = 30I''+15+2d''$, $I'' \geq 1$ and $0 \leq d'' < 15$.
This theorem is the union of the result from Lemma \ref{high11} and Lemma \ref{high15}. 
The case for $I'' \geq 1$ requires the knowledge that $\{n-23\}\cup [n-14, n-12] \subseteq [2I'' + 15+2d'', n] \setminus (\{n-29\} \cup [n-19, n-15] \cup [ n-9, n-1])$, which is easily seen as $2I'' + 15+2d'' \leq n-23 = 30I''+15+2d'' -23$ when $I'' \geq 1$. 
Then the union of $[11+2d', n] \setminus (\{n-23\}\cup [n-15, n-12] \cup [n-7, n-1])$ and $[2I'' + 15+2d'', n] \setminus  (\{n-29\} \cup [n-19, n-15] \cup [ n-9, n-1])$ gives the result as stated in the theorem when $n \geq 45$. 
The case for $33 \leq n< 45$ is covered by Lemma \ref{high11}.
\end{proof}

\begin{lemma}
\label{lower22}
For odd $n \geq 33$, let $I'$ and $d'$ be the unique integers such that $n = 22I'+11+2d'$, $I' \geq 1$ and $0 \leq d' < 11$. Then there exist four transversals of $B_n$ that intersect stably in $t$ points for $t\in [3+d', 22I'+3+d'] \setminus (\{22I'-20+d'\}\cup [22I'-12+d', 22I'-9+d']\cup [22I'-4+d', 22I'+2+d'])$.

\end{lemma}
\begin{proof}
Using  Theorem \ref{mainTheorem}, Lemma \ref{low}, and Lemma \ref{lemma2I}, we can conclude there exists the required set of transversals for each $t\in  \{d'\}+\{3\} + 2I' {\Omega}_4^{11}(11) \subseteq \{d\}+ {\Omega}_4^{11}(11+d') + 2I'  {\Omega}_4^{11}(11) $, and hence for each $t\in\{3+d', \ldots, 22I'+3+d'\} \setminus \{22I'-20+d', 22I'-12+d', \ldots, 22I'-9+d', 22I'-4+d', \ldots, 22I'+2+d'\}$.
\end{proof}

\begin{lemma}
\label{lower18}
For odd $n \geq 27$, let  $I$ and $d$ be the unique integers such that $n = 18I+9+2d$, $I \geq 1$ and $0 \leq d < 9$. Then there exist four transversals of $B_n$ that intersect stably in $t$ points for $t\in [d, \ldots, 18I+d] \setminus \bigcup_{i=1}^{7}[18I-8i+1+d,18I-9i+8+d]$.
\end{lemma}
\begin{proof}
Using Theorem \ref{mainTheorem}, Lemma \ref{low}, and Lemma \ref{lemma2I}, we can conclude there exists the required set of transversals for each $t\in  \{d\}+\{0\} + 2I {\Omega}_4^{9}(9) \subseteq \{d\}+ {\Omega}_4^{9}(9+d) + 2I {\Omega}_4^{9}(9) $, and hence for each $t\in[d, 18I+d] \setminus \bigcup_{i=1}^{7}[18I-8i+1+d,18I-9i+8+d]$.
\end{proof}

\begin{theorem} \label{thm-low}
For odd $n \geq 33$, let  $I$, $I'$, $d$ and $d'$ be the unique integers such that $n= 18I +9 + 2d$ and $n = 22I'+11+2d'$, $I,I' \geq 1$, $0 \leq d < 9$ and $0 \leq d' < 11$. Then there exist four transversals of $B_n$ that intersect stably in $t$ points for $t\in [min(3+d',d), 11+2d']$ except, perhaps, when:
\begin{itemize}
\item $33 \leq n \leq 43$ and $t \in [10+d',11+d']$,
\item $45 \leq n \leq 53$ and $t \in [-1+d', 2+d'] \cup [10+d',11+d'] \cup [18+d', 20+d']$,
\item $63 \leq n \leq 75$ and  $t \in [7+d, 8+d]$.
\end{itemize}
\end{theorem}
\begin{proof}
For $33 \leq n \leq 43$, we have $I'=I=1$ and $min(3+d',d) = 3+d' = d$. 
Then we have the existance of four transversal of $B_n$ that intersect stably in $t$ points by Lemma \ref{lower22} for $t\in [3+d', 25+d'] \setminus ([10+d', 13+d'] \cup [18+d', 24+d'])$ and by Lemma \ref{lower18} for $t\in [d, 18+d] \setminus ([3+d, 8+d] \cup [11+d, 17+d]) = [3+d', 21+d'] \setminus ([6+d', 11+d'] \cup [14+d', 20+d'])$. 
The union of the two result sets is $[3+d', 25+d'] \setminus ([10+d', 11+d'] \cup [18+d', 20+d'] \cup [22+d', 24+d'])$. 
A subset of this is $[3+d', 16+d'] \setminus [10+d', 11+d']$. 
Noting that as $d' \leq 5$ for the specified $n$, then $ 11+2d' \leq 16+d'$, and so this subset includes the range $ [min(3+d',d), 11+2d']\setminus [10+d', 11+d']$, which is the required result when $33 \leq n \leq 43$.

For $45 \leq n \leq 53$, we have $I=2$, $I'=1$ and as $d' = 6+d$ we have $min(3+d',d)  = d$. 
Then we have the existence of four transversal of $B_n$ that intersect stably in $t$ points by Lemma \ref{lower22} for $t\in [3+d', 25+d'] \setminus ([10+d', 13+d'] \cup [18+d', 24+d'])$
and by Lemma \ref{lower18} for $[d, 36+d] \setminus ([5+d, 8+d] \cup [13+d, 17+d] \cup [21+d, 26+d] \cup [29+d, 35+d]) = [-6+d', 30+d'] \setminus ([-1+d', 2+d'] \cup [7+d', 11+d'] \cup [15+d', 20+d'] \cup [23+d', 29+d'])$. 
The union of the two result sets is $[d, \ldots, 30+d'] \setminus ([-1+d', 2+d']\cup [10+d', 11+d'] \cup [18+d', 20+d'] \cup [23+d', 24+d'] \cup [26+d',29+d'])$. 
A subset of this is $[d, 21+d'] \setminus ([-1+d', 2+d']\cup [10+d', 11+d'] \cup [18+d', 20+d'])$. 
Noting that as $d' \leq 10$, then $ 11+2d' \leq 21+d'$, and so this subset includes the range $ [min(3+d',d), 11+2d']\setminus ([-1+d', 2+d']\cup [10+d', 11+d'] \cup [18+d', 20+d'])$, which is the required result when $45 \leq n \leq 53$.

For $n \geq 55$, we have $I' \geq 2$. 
Then we have the existence of four transversal of $B_n$ that intersect stably in $t$ points by Lemma \ref{lower22} for $t\in [3+d', 21+d']$. 
This completes the case when $min(3+d',d)=3+d'$, as $11+2d' \leq 21+d'$. 
When $d<3+d'$, we still need the cases $t\in [d, \ldots, 3+d']$. 
As $3+d' \leq 13$, it is enough to show the statement holds for those $t$ with $t \in [d, 13]$.

When $I=3$, then $63 \leq n \leq 75$, and we have the existance of four transversal of $B_n$ that intersect stably in $t$ points by Lemma \ref{lower18} for $t\in [d, 13] \setminus [7+d,8+d]$. When $I\geq 4$, we have the existence of four transversal of $B_n$ that intersect stably in $t$ points by Lemma \ref{lower18} for $t\in [d, 13]$.
\end{proof}

Then Theorem \ref{overall4} follows by Theorem \ref{thm-high} and Theorem \ref{thm-low}.

\section{Application to latin trades}

Let $D$ represent a combinatorial design and assume there exists distinct sets $S_1,S_2$ with $S_1 \subseteq D$, such that $D' =(D \setminus S_1) \cup S_2$ forms a valid design. 
Then the pair $(S_1,S_2)$ forms a \emph{combinatorial bitrade}. The original design $D$ is immaterial, and we can define a bitrade formally by taking the pair of sub-designs $(S_1,S_2)$ that fulfill certain properties. If our combinatorial design is a latin square, the bitrade is called a latin bitrade. A good survey of latin bitrades is \cite{CavSurvey}, and for trades in general is \cite{TradesSurvey}.

 \begin{definition}
\label{def:trade}
A \emph{$\mu$-way latin trade} of volume $s$ and order $n$ is a collection of $\mu$ partial latin squares $(L_1, \ldots, L_{\mu})$, each of order $n$, such that:
\begin{enumerate}
\item Each partial latin square contains exactly the same $s$ filled cells,
\item If cell $(i,j)$ is filled then it contains a different entry in each of the $\mu$ partial latin squares,
\item Row $i$ in each of the $\mu$ partial latin squares contains, set-wise, the same symbols, and column $j$ likewise.
\end{enumerate}
\end{definition}

A $\mu$-way latin trade is \emph{circulant} if each of the partial latin squares can be obtained from the first row by simultaneously cycling the rows, columns, and symbols. 
For example, the cell $(r,c,e)\in L$ would imply $(r+1,c+1,e+1)\in L$. 
We call the set of first rows the \emph{base row}, and can write it in the notation $B=\{(e_1, \ldots, e_{\mu})_{c_j} \mid 1 \leq j \leq k\}$, where $(0,c_j,e_{\alpha}) \in L_{\alpha}$ for $1 \leq \alpha \leq \mu$.

A $\mu$-way latin trade is \emph{$k$-homogeneous} if in each partial latin square, $L$, each row and each column contain $k$ filled cells, and each symbol appears in filled cells of $L$ precisely $k$ times.
Clearly a circulant $\mu$-way trade is $k$-homogeneous, where $k$ is the number of filled cells in the first row.

There has been much interest in $2$-way $k$-homogeneous latin trades as demonstrated by the work in \cite{khomo}, \cite{Exiskhomo}, \cite{cav3-homo}, \cite{3-homo}, \cite{4-homo}, and \cite{Perm34Homo}, and more recently there has been an extension to  $\mu$-way $k$-homogeneous latin trades in \cite{3way}.

\begin{theorem} \label{trans2trade}
If there exists a set of $\mu$ transversals of $B_n$ that intersect stably in $t$ points, then there exists a circulant $\mu$-way $(n-t)$-homogeneous latin trade of order $n$.
\end{theorem}
\begin{proof}
Consider a  set of $\mu$ transversals of $B_n$, $T_1, \ldots, T_{\mu}$, that intersect stably in the $t$ points $S$. 
Consider the partial latin squares $Q_\alpha = \{(i,c+i, r+c+i) \mid 0 \leq i \leq n-1 \text{ and } (r,c,r+c) \in T_\alpha \setminus S\}$. 
It is clear that each corresponding row of the $Q_\alpha$ contain setwise the same symbols. As the cells of the first column of $Q_\alpha$ are $(-c,0,r) \in Q_\alpha$, each column contain setwise the same symbols. 
Then it is clear that the collection of $\mu$ partial latin squares satisfy the conditions of a $\mu$-way latin trade. 
They are also circulant by definition, and hence are clearly $(n-t)$-homogeneous.
\end{proof}

\begin{example}

Consider $B_5$ with the following transversals:

 \begin{center}
  \begin{tabular}{ |c | c | c| c | c | c|  }
    \hline
$\underline{0}^*$&$1$&$2$&$3$&$4$\\
          \hline
$1$&$\underline{2}$&$3^*$&$4$&$0$\\
          \hline
${2}$&$3$&$\underline{4}$&$0$&$1^*$\\
          \hline
$3$&$4^*$&$0$&$\underline{1}$&$2$\\
          \hline
$4$&$0$&$1$&$2^*$&$\underline{3}$\\
          \hline
  \end{tabular}
\end{center}

Here, the transversals intersect stably in the $1$ point $ S=\{(0,0,0)\}$.
The cell $(1,2,3)$ is in the starred transversal, and not in $S$, so Construction \ref{trans2trade} places the cell $(0,2,3)$ into the resulting first row of a circulant latin square.
Construction \ref{trans2trade} gives the first row of a circulant latin squares to be:

 \begin{center}
  \begin{tabular}{ |c | c | c| c | c | c|  }
    \hline
$\cdot$&$4$&$3$&$2$&$1$\\
          \hline
  \end{tabular} \hspace{10pt}
  \begin{tabular}{ |c | c | c| c | c | c|  }
    \hline
$\cdot$&$2$&$4$&$1$&$3$\\
          \hline
  \end{tabular}

\end{center}

Writing these latin squares out completely:

 \begin{center}
  \begin{tabular}{ |c | c | c| c | c | c|  }
    \hline
$\cdot$&$4$&$3$&$2$&$1$\\
          \hline
$2$&$\cdot$&$0$&$4$&$3$\\
          \hline
$4$&$3$&$\cdot$&$1$&$0$\\
          \hline
$1$&$0$&$4$&$\cdot$&$2$\\
          \hline
$3$&$2$&$1$&$0$&$\cdot$\\
          \hline
  \end{tabular} \hspace{10pt}
  \begin{tabular}{ |c | c | c| c | c | c|  }
    \hline
$\cdot$&$2$&$4$&$1$&$3$\\
          \hline
$4$&$\cdot$&$3$&$0$&$2$\\
          \hline
$3$&$0$&$\cdot$&$4$&$1$\\
          \hline
$2$&$4$&$1$&$\cdot$&$0$\\
          \hline
$1$&$3$&$0$&$2$&$\cdot$\\
          \hline
  \end{tabular}
\end{center}

The two partial latin squares form a $2$-way $4$-homogeneous circulant latin trade of order $5$. 
This completes the example.

\end{example}

\begin{theorem} \label{trades3}
For odd $n \geq 33$, let $I$, $I'$, $d$ and $d'$ be the unique integers such that $n= 18I +9 + 2d$ and $n = 22I'+11+2d'$, $I,I' \geq 1$, $0 \leq d < 9$ and $0 \leq d' < 11$. Then there exists a cyclic $(n-t)$-homogeneous $3$-way latin trade of order $n$, for $t\in [min(3+d',d),  n] \setminus [n-5, n-1]$, except, perhaps, when:

\begin{itemize}
\item $n=51$ and $t=29$,
\item $n=53$ and $t=30$.
\end{itemize}
\end{theorem}
\begin{proof}
Follows by Theorem \ref{overall3} and Theorem \ref{trans2trade}.
\end{proof}

\begin{theorem}\label{trades4}
For odd $n \geq 33$, let $I$, $I'$, $d$ and $d'$ be the unique integers such that $n= 18I +9 + 2d$ and $n = 22I'+11+2d'$, $I,I' \geq 1$, $0 \leq d < 9$ and $0 \leq d' < 11$. Then there exists a cyclic $(n-t)$-homogeneous $4$-way latin trade of order $n$, for $t\in [min(3+d',d),  n] \setminus (\{n-15\} \cup [n-7, n-1])$, except, perhaps, when:
\begin{itemize}
\item $33 \leq n \leq 43$ and $t \in [10+d',11+d'] \cup [ n-14, n-12]$,
\item $45 \leq n \leq 53$ and $t \in [-1+d', 2+d'] \cup [10+d',11+d'] \cup [18+d', 20+d']$,
\item $63 \leq n \leq 75$ and  $t \in [7+d, 8+d]$.

\end{itemize}
\end{theorem}
\begin{proof}
Follows by Theorem \ref{overall4} and Theorem \ref{trans2trade}.
\end{proof}

\section{Conclusion and future work}

We have been able to show, with a number of exceptions, that their exists three (resp. four) transversals of $B_n$ that intersect stably in $t$ points when $n$ is odd and $n \geq 33$. 
With few unsolved cases left, it appears that future work may be able to answer \ref{Qn:main} completely for $\mu=3,4$. 
%
%
%
%

Theorem \ref{trades3} and \ref{trades4} fill in a large portion of the spectrum of $3$/$4$-way $k$-homogeneous latin trades of odd order, which is a significant advancement on what was previously known.  
There are a number of construction for $\mu$-way $k$-homogeneous latin trades \cite{3way}, and it seems that further work may result in the spectrum being completed.

\section{Acknowledgments}

We would like to acknowledge Ian Wanless for suggesting Lemma \ref{lem:wanlessIdea}, and Diane Donovan for proof-reading this document and providing valuable feedback.

\newpage 
\appendix

\section{Base $b=9$}

\begin{table}[!h]
\begin{center}
  \begin{tabular}{ |c | c | c|  }
    \hline
$b+d$&intersect&result\\ \hline

\hline
9&0&
8 6 7 2 0 1 5 3 4 
\\ \cline{3-3} 
&&7 8 6 1 2 0 4 5 3 
\\ \cline{3-3} 
&&5 3 4 8 6 7 2 0 1 
\\ \cline{3-3} 
&&4 5 3 7 8 6 1 2 0 
\\ \hline

9&1&
7 8 6 2 0 1 4 5 3 
\\ \cline{3-3} 
&&8 5 3 7 0 6 1 2 4 
\\ \cline{3-3} 
&&6 4 5 8 0 7 3 1 2 
\\ \cline{3-3} 
&&5 7 8 3 0 2 6 4 1 
\\ \hline

10&0&

8 9 5 2 0 1 6 7 3 4 
 \\ \cline{3-3}
&&7 5 8 1 9 0 2 4 6 3 
 \\ \cline{3-3}
&&6 4 2 9 7 3 8 0 5 1 
 \\ \cline{3-3}
&&5 3 4 7 8 9 1 6 0 2 
\\ \hline

11&0&
10 7 5 2 0 1 9 6 8 3 4 
 \\ \cline{3-3}
&&7 9 6 1 2 0 10 8 4 5 3 
 \\ \cline{3-3}
&&8 4 2 10 3 9 0 5 7 1 6 
 \\ \cline{3-3}
&&6 3 10 7 1 2 8 9 0 4 5 
\\ \hline

12&0&
10 7 5 2 0 1 9 11 8 3 4 6 
\\ \cline{3-3}
&&7 11 6 1 2 0 10 3 9 5 8 4 
\\ \cline{3-3}
&&8 9 2 4 1 11 0 5 10 6 7 3 
\\ \cline{3-3}
&&5 3 4 11 6 10 1 9 0 8 2 7 
\\ \hline

13&0&
12 7 5 2 0 1 11 3 8 10 4 9 6 
\\ \cline{3-3}
&&8 6 12 1 2 0 4 11 9 3 10 5 7 
\\ \cline{3-3}
&&10 11 2 4 1 3 0 7 12 9 6 8 5 
\\ \cline{3-3}
&&5 3 4 11 12 7 2 0 10 1 9 6 8
\\ \hline
14&0&
12 7 5 2 0 1 13 3 8 11 4 10 6 9 
\\ \cline{3-3}
&&8 6 12 1 2 0 4 11 13 3 9 5 10 7 
\\ \cline{3-3}
&&10 13 2 4 1 3 0 5 12 9 6 11 7 8 
\\ \cline{3-3}
&&5 3 4 7 12 9 2 0 10 13 11 1 8 6 
\\ \hline

15&0&
14 7 5 2 0 1 12 3 4 13 6 10 8 11 9 
\\ \cline{3-3}
&&8 6 14 1 2 0 4 7 13 3 12 9 11 5 10 
\\ \cline{3-3}
&&12 9 2 4 1 3 0 13 8 5 14 11 6 10 7 
\\ \cline{3-3}
&&5 3 4 9 10 11 2 0 14 1 13 7 12 8 6 
\\ \hline

16&0&
14 7 5 2 0 1 12 3 4 15 6 11 8 13 9 10 
\\ \cline{3-3}
&&8 6 14 1 2 0 4 7 15 3 12 9 13 5 10 11 
\\ \cline{3-3}
&&12 9 2 4 1 3 0 13 8 5 14 15 6 10 11 7 
\\ \cline{3-3}
&&5 3 4 11 12 7 2 0 10 1 15 13 14 9 6 8 
\\ \hline

\end{tabular}
\caption{$\mu=4$ and $b=9$.}
  \label{tab:tab4-9a}
\end{center}
\end{table}

\begin{table}[h!]
\begin{center}
  \begin{tabular}{ |c | c | c|  }
    \hline
$b+d$&intersect&result\\ \hline

17&0&
16 7 5 2 0 1 8 3 4 15 12 9 6 13 14 10 11 
\\ \cline{3-3}
&&8 6 16 1 2 0 4 7 14 3 10 5 15 11 12 13 9 
\\ \cline{3-3}
&&14 9 2 4 1 3 0 5 16 13 6 7 8 15 11 12 10 
\\ \cline{3-3}
&&5 3 4 15 10 7 2 0 8 1 14 16 13 9 6 11 12 
\\ \hline

\end{tabular}
\caption{$\mu=4$ and $b=9$.}
  \label{tab:tab4-9b}
\end{center}
\end{table}

\begin{table}[!h]
\begin{center}
  \begin{tabular}{ |c | c | c|  }
\hline
$b+d$&intersect&result\\ \hline

9&2&
8 6 7 2 0 1 5 3 4 
\\ \cline{3-3} 
&&7 5 8 2 0 3 6 4 1 
\\ \cline{3-3} 
&&6 8 5 2 0 7 4 1 3
\\ \hline
 
9&3&
7 5 8 2 0 3 6 4 1 
\\ \cline{3-3} 
&&6 3 8 2 4 7 5 0 1 
\\ \cline{3-3} 
&&4 7 8 2 3 6 0 5 1 
\\ \hline
 \end{tabular}
\caption{$\mu=3$ and $b=9$.}
 \label{tab:tab3-9a}
\end{center}
\end{table}

\newpage

\section{Base $b=11$}

\begin{table}[!h]
\begin{center}
  \begin{tabular}{ |c | c | c|  }
    \hline
$b+d$&intersect&result\\ \hline

11&0&
10 8 9 5 2 0 1 6 7 3 4 
\\ \cline{3-3}
&&9 10 6 7 1 2 0 8 4 5 3 
\\ \cline{3-3}
&&8 9 10 2 7 1 3 0 6 4 5 
\\ \cline{3-3}
&&7 5 3 10 8 9 4 2 0 6 1
\\ \hline

11&1&
10 8 9 5 2 0 1 6 7 3 4 
\\ \cline{3-3}
&&9 7 5 3 10 0 6 8 2 4 1 
\\ \cline{3-3}
&&8 9 4 10 3 0 5 7 1 6 2
\\ \cline{3-3} 
&&7 5 10 8 4 0 9 2 6 1 3
\\ \hline

11&2& 
10 8 9 5 2 0 1 6 7 3 4
\\ \cline{3-3} 
&&9 7 10 4 2 0 5 8 6 1 3
\\ \cline{3-3} 
&&8 9 5 10 2 0 3 7 4 6 1
\\ \cline{3-3} 
&&7 10 6 9 2 0 8 3 1 4 5
\\ \hline 
11&3&
8 9 10 4 2 0 5 7 1 6 3
\\ \cline{3-3} 
&&8 10 5 2 9 7 0 3 1 6 4
\\ \cline{3-3} 
&&8 5 3 10 7 9 4 0 1 6 2
\\ \cline{3-3} 
&&8 4 9 3 10 2 7 5 1 6 0
\\ \hline

  \end{tabular}
\caption{$\mu=4$ and $b=11$.}
  \label{tab:tab4-11a}
\end{center}
\end{table}

\begin{table}[!h]
\begin{center}
  \begin{tabular}{ |c | c | c|  }
    \hline
$b+d$&intersect&result\\ \hline

12&3&
9 11 6 7 2 0 1 10 8 5 3 4 
\\ \cline{3-3} 
&&8 5 7 2 9 11 1 10 6 3 0 4 
\\ \cline{3-3} 
&&6 9 3 11 5 8 1 10 0 7 2 4 
\\ \cline{3-3} 
&&5 7 11 3 8 9 1 10 2 0 6 4
\\ \hline

13&3&
12 9 7 5 2 0 1 10 11 6 8 3 4 
\\ \cline{3-3} 
&&7 9 6 3 1 12 10 2 11 5 8 4 0 
\\ \cline{3-3} 
&&6 9 12 2 4 10 3 0 11 7 8 1 5 
\\ \cline{3-3} 
&&5 9 10 4 12 1 2 7 11 0 8 6 3
\\ \hline

14&3&
12 13 7 5 2 0 1 3 11 9 10 6 4 8 
\\ \cline{3-3} 
&&10 13 6 4 2 0 11 12 1 3 8 5 9 7 
\\ \cline{3-3} 
&&8 13 5 9 2 0 3 11 12 1 6 10 7 4 
\\ \cline{3-3} 
&&7 13 8 6 2 0 12 1 4 10 11 9 5 3 
\\ \hline

15&3&
14 11 7 5 2 0 1 3 12 13 8 10 4 6 9 
\\ \cline{3-3} 
&&8 9 14 6 2 0 1 13 11 3 4 12 10 5 7 
\\ \cline{3-3} 
&&10 8 6 11 2 0 1 5 14 12 13 9 7 3 4 
\\ \cline{3-3} 
&&9 7 8 13 2 0 1 14 4 11 12 3 6 10 5
\\ \hline

16&3& 
14 15 7 5 2 0 1 3 4 13 11 9 6 12 10 8 
\\ \cline{3-3} 
&&12 8 6 7 2 0 1 15 13 14 10 3 4 5 11 9 
\\ \cline{3-3} 
&&9 7 8 13 2 0 1 14 6 3 15 12 10 11 4 5 
\\ \cline{3-3} 
&&8 9 14 6 2 0 1 5 12 15 4 11 13 10 7 3 
\\ \hline

17&3&
16 13 7 5 2 0 1 3 4 14 15 9 6 11 8 12 10 
\\ \cline{3-3} 
&&12 8 6 7 2 0 1 16 14 15 10 3 4 5 13 11 9 
\\ \cline{3-3} 
&&9 7 8 15 2 0 1 13 16 3 4 5 11 14 12 10 6 
\\ \cline{3-3} 
&&8 9 12 6 2 0 1 5 15 7 14 16 10 13 11 3 4
\\ \hline

18&3&
16 13 7 5 2 0 1 3 4 17 15 9 6 11 8 14 12 10 
\\ \cline{3-3} 
&&12 8 6 7 2 0 1 17 14 16 10 3 4 5 15 13 11 9 
\\ \cline{3-3} 
&&9 7 8 17 2 0 1 11 16 3 4 5 15 13 14 10 6 12 
\\ \cline{3-3} 
&&8 9 12 6 2 0 1 5 10 15 16 14 17 3 13 7 4 11 
\\ \hline

19&3&
18 13 7 5 2 0 1 3 4 17 6 16 8 9 10 14 15 11 12 
\\ \cline{3-3} 
&&12 8 6 7 2 0 1 15 16 18 10 3 4 5 17 13 14 9 11 
\\ \cline{3-3} 
&&9 7 8 17 2 0 1 11 18 3 4 5 12 16 13 15 6 14 10 
\\ \cline{3-3} 
&&8 9 12 6 2 0 1 5 14 15 18 7 17 3 16 11 4 10 13
\\ \hline
20&3&
18 13 7 5 2 0 1 3 4 19 6 15 8 9 10 17 14 16 11 12 
\\ \cline{3-3} 
&&12 8 6 7 2 0 1 19 14 11 18 3 4 5 17 9 16 13 15 10 
\\ \cline{3-3} 
&&9 7 8 17 2 0 1 15 16 3 4 5 6 19 12 18 13 14 10 11 
\\ \cline{3-3} 
&&8 9 12 6 2 0 1 5 18 15 10 7 19 3 16 14 17 11 4 13 
\\ \hline

21&3&
20 13 7 5 2 0 1 3 4 9 6 19 16 11 8 18 10 17 14 12 15 
\\ \cline{3-3} 
&&12 8 6 7 2 0 1 19 20 13 10 3 4 5 18 9 15 16 17 11 14 
\\ \cline{3-3} 
&&9 7 8 19 2 0 1 17 12 3 4 5 6 18 20 11 14 15 10 16 13 
\\ \cline{3-3} 
&&8 9 12 6 2 0 1 5 14 11 18 15 19 3 4 20 17 7 16 13 10 
\\ \hline
  \end{tabular}
\caption{$\mu=4$ and $b=11$.}
  \label{tab:tab4-11b}
\end{center}
\end{table}

\begin{table}[!h]
\begin{center}
  \begin{tabular}{ |c | c | c|  }
\hline
$b+d$&intersect&result\\ \hline

\hline
11&4&
9 10 8 5 2 0 1 6 7 3 4 
\\ \cline{3-3} 
&&10 6 7 8 2 0 9 1 5 3 4 
\\ \cline{3-3} 
&&7 9 6 10 2 0 5 8 1 3 4 
\\ \hline
11&5&
10 8 6 9 2 0 1 4 7 5 3 
\\ \cline{3-3} 
&&10 7 9 4 2 0 6 8 1 5 3 
\\ \cline{3-3} 
&&10 6 7 8 2 0 9 1 4 5 3 
\\ \hline
 \end{tabular}
\caption{$\mu=3$ and $b=11$.}
\label{tab:tab3-11a}
\end{center}
\end{table}

\newpage

\section{Base $b=15$}

\begin{table}[!h]
\begin{center}
  \begin{tabular}{ |c | c | c|  }
    \hline
$b+d$&intersect&result\\ \hline

15&1&
14 12 13 9 7 5 2 0 1 10 11 6 8 3 4 
\\ \cline{3-3} 
&&13 14 12 8 5 3 11 0 10 1 2 9 4 6 7 
\\ \cline{3-3} 
&&12 13 11 14 4 10 3 0 2 9 1 5 7 8 6 
\\ \cline{3-3} 
&&11 9 14 5 8 4 13 0 6 12 10 7 1 2 3 
\\ \hline
15&2& 
14 12 13 9 7 5 2 0 1 10 11 6 8 3 4 
\\ \cline{3-3} 
&&13 14 10 11 12 4 2 0 3 1 9 7 5 8 6 
\\ \cline{3-3} 
&&12 13 14 7 5 8 2 0 10 11 1 4 9 6 3 
\\ \cline{3-3} 
&&11 9 12 6 14 10 2 0 13 8 3 1 4 7 5 
\\ \hline

15&3&
14 12 13 9 7 5 2 0 1 10 11 6 8 3 4 
\\ \cline{3-3} 
&&13 14 12 8 6 7 2 0 1 11 9 10 4 5 3 
\\ \cline{3-3} 
&&11 13 10 7 14 12 2 0 1 4 6 9 3 8 5 
\\ \cline{3-3} 
&&12 10 8 13 11 14 2 0 1 5 3 7 9 4 6 
\\ \hline

15&4& 
14 12 13 9 7 5 2 0 1 10 11 6 8 3 4 
\\ \cline{3-3} 
&&13 11 14 8 10 5 2 0 1 12 9 7 3 4 6 
\\ \cline{3-3} 
&&12 14 9 10 13 5 2 0 1 11 8 3 4 6 7 
\\ \cline{3-3} 
&&11 13 10 12 14 5 2 0 1 4 6 9 7 8 3 
\\ \hline

15&5& 
14 12 13 9 7 5 2 0 1 10 11 6 8 3 4 
\\ \cline{3-3} 
&&13 11 14 8 10 5 2 0 1 12 7 9 3 6 4 
\\ \cline{3-3} 
&&12 14 9 10 13 5 2 0 1 11 6 3 7 8 4 
\\ \cline{3-3} 
&&11 13 10 14 12 5 2 0 1 6 3 8 9 7 4 
\\ \hline
  \end{tabular}

\caption{$\mu=4$ and $b=15$.}
  \label{tab:tab4-15a}
\end{center}
\end{table}

\newpage \ 

\begin{table}[!h]
\begin{center}
  \begin{tabular}{ |c | c | c|  }
    \hline
$b+d$&intersect&result\\ \hline

15&15&
14 12 13 9 7 5 2 0 1 10 11 6 8 3 4 
\\ \hline
16&16&
14 12 15 13 7 5 2 0 1 3 11 9 10 6 4 8 
\\ \hline
17&17&
14 15 16 10 7 5 2 0 1 3 13 11 12 8 6 4 9 
\\ \hline
18&18&
16 17 12 10 7 5 2 0 1 3 15 13 11 14 8 6 4 9 
\\ \hline
19&19&
16 17 18 10 7 5 2 0 1 3  \ldots\\ && 4 15 13 11 14 8 6 12 9 
\\ \hline
20&20&
18 19 14 10 7 5 2 0 1 3  \ldots\\ && 4 17 15 16 12 9 6 8 13 11 
\\ \hline
21&21&
20 17 14 10 7 5 2 0 1 3  4  \ldots\\ &&18 19 15 13 9 6 16 8 11 12 
\\ \hline
22&22&
20 21 14 10 7 5 2 0 1 3 4  \ldots\\ && 19 6  16 18 13 8 9 17 15 11 12 
\\ \hline
23&23&
22 19 14 10 7 5 2 0 1 3 4 21  \ldots\\ && 6 20 17 13 8 9 18 11 15 16 12 
\\ \hline
24&24&
22 23 14 10 7 5 2 0 1 3 4 9 6  \ldots\\ && 21 19 15 20 11 8 18 12 17 13 16 
\\ \hline
25&25&
24 21 14 10 7 5 2 0 1 3 4 9 6 22  \ldots\\ &&23 15  20 11 8 19 12 13 18 16 17 
\\ \hline
26&26&
24 25 14 10 7 5 2 0 1 3 4 9 6 23 8  \ldots\\ &&22 18 11 12 13 20 21 17 15 19 16 
\\ \hline
27&27&
26 23 14 10 7 5 2 0 1 3 4 9 6 25 8  \ldots\\ &&24 18 11 12 13 22 15 19 21 16 20 17 

\\ \hline
28&28&
26 27 14 10 7 5 2 0 1 3 4 9 6 11 8 25 \ldots\\ &&22 17 12 13 24 15 23 20 18 16 21 19 

\\ \hline
29&29&
28 25 14 10 7 5 2 0 1 3 4 9 6 11 8 27  \ldots\\ &&24 17 12 13 26 15 16 22 20 18 23 21 19 
\\ \hline
 \end{tabular}

\caption{$\mu=4$ and $b=15$ (Only one partial transversal is needed in each case).}
  \label{tab:tab4-15b}
\end{center}
\end{table}

\newpage \ 

\section*{References}

\bibliographystyle{elsarticle-harv} 
\bibliography{ref}
\end{document}